\documentclass{article}
\usepackage{amsthm,amsmath}
\usepackage{color,amsmath,amsthm,amsopn,amstext,amscd,amsfonts,amssymb}
\usepackage{epsfig}
{}

\theoremstyle{plain}
\newtheorem{thm}{Theorem}
\definecolor{c20}{rgb}{1,0,.2}
\definecolor{c10}{rgb}{0,0,1}

\newcommand{\iid}{i.i.d.}

\begin{document}
\title{Generalized extreme shock models with a possibly increasing threshold}

\author{Pasquale Cirillo\footnote{Corresponding author, Institute of Mathematical Statistics and Actuarial Sciences, University of Bern, Sidlerstrasse 5, CH-3012, Bern. Mail to: pasquale.cirillo@stat.unibe.ch. Tel: +41 (0)31 631 88 03. Fax: +41 (0)31 631 38 70 c/o Cirillo} \,and J\"urg H\"usler\\ 
Institute of Mathematical Statistics and Actuarial Sciences,\\University of Bern,\\
Sidlerstrasse 5, CH-3012, Bern\\
\texttt{http://imsv.unibe.ch}
}

\maketitle

\begin{abstract}
We propose a generalized extreme shock model with a possibly increasing failure threshold.\\
While standard models assume that the crucial threshold for the system may only decrease over time, because of weakening shocks and obsolescence, we assume that, especially at the beginning of the system's life, some strengthening shocks may increase the system tolerance to large shock. This is for example the case of turbines' running-in in the field of engineering.\\
On the basis of parametric assumptions, we provide theoretical results and derive some exact and asymptotic univariate and multivariate distributions for the model.\\
In the last part of the paper we show how to link this new model to some nonparametric approaches proposed in the literature.

\end{abstract}
\newpage

\section{Introduction}

The setup in extreme shock models is a family $\{(X_k,Y_k),
k\ge0\}$ of independent identically distributed (\iid) two-dimensional random vectors, with $X_k$
the magnitude of the $k$\,th shock and $Y_k$ the time between
the $(k-1)$\,th and the $k$\,th shock. The main object of
interest is the lifetime/failure time of the system by assuming certain schemes for the failure. These models are motivated by the possible breakdown of a material or of a system subject to random shocks of random magnitude, as it occurs in engineering. Anyway, it is easy to see useful applications also in other fields, such as economics, medicine and biology.\\
Cumulative shock models and extreme shock models are discussed, as well as mixtures of both models, in
Gut and H\"usler (1999, 2005) and the references therein.\\
In the cumulative shock model we consider  $$T_n=\sum_{k\le n} Y_k \quad \mbox{ and } \quad S_n=\sum_{k\le n} X_k,\,
$$ for $n\ge1$, with $T_0=S_0=0$. The failure of the system occurs if $S_n>\alpha $ for some $n$ and $\alpha $. Here $\alpha$ denotes the critical threshold of the system. The time until the system fails the first time or the failure time $T_\tau$ with  $\tau=\min\{n: S_n>\alpha\}$ are then of interest. For results see Gut (1990).\\
In the simple extreme shock model one large or extreme shock, larger than a given failure (or crucial) threshold $\gamma$, may cause the default of the system. The lifetime of the system is in this case defined as $T_{\nu}$
where
\begin{equation}
\nu=\min \{n:X_n>\alpha\}.
\end{equation}
This model was dealt with in Gut and H\"usler (1999), and Gut (2001). \\
Gut and H\"usler (2005) extended this simple model to a more realistic framework by assuming that the failure threshold is not constant, but that it may vary with time, depending on the experienced shocks. In detail they assume that large but not fatal shock may effect system's
tolerance to subsequent shocks, because of cracks in the structure for example. To be more exact, for a fixed $\alpha_0>0$ a shock
$X_{i}$ can damage the system if it is larger than a certain boundary value
$\beta<\alpha_0$\footnote{The value $\beta$ can also vary over time. The only requirement is that it is always lower than the corresponding failure threshold.}. As long as $X_{i}<\alpha_0$ the system does not fail. The crucial
hypothesis is the following: if a first nonfatal shock comes with values in
$\left[ \beta,\alpha_0\right]  $ the maximum load limit of the system is no more
$\alpha_0$ but decreases to $\alpha_{1}\in\left[  \beta,\alpha_1\right]  $. At this
point, if another large but not too strong shock occurs in $\left[  \beta,\alpha_0\right]  $, the new crucial threshold is lowered again to
$\alpha_{2}\in\left[\beta,\alpha_1\right]  $ and so on until the
system fails. We could call all this \textquotedblleft risky threshold
mechanism\textquotedblright. Naturally, $\forall t$%
\begin{equation}
\alpha_0\geq\alpha_1 \geq\alpha_2\geq...\geq\beta.
\end{equation}
Hence one can define the stopping time $\nu=\min\left\{  n:X_{n}\geq
\alpha_{L(n-1)}\right\}  $ with $L(n)=\sum_{i=1}^{n}1_{\left\{
X_{i}\geq\beta \right\}  }$ and $L(0)=0$.
Gut and H\"{u}sler (2005) have shown that the results for generalized extreme shock models (GEMS) are identical to
the simple extreme case for nonrandom $\alpha_{k}$, while this is not true in the random case. We refer to the original paper for more details.\\
Even if the modeling of GESM is surely sensible, sometimes it can be worth to consider a default threshold $\alpha$ that might even increase, at least initially, say in particular for a running-in period of some equipment. This is for example the case of turbines' breaking-in in the field of engineering, but other applications can be found in electric networks and biology (e.g. Siphonophora in their growing process, Dunn et al., 2005). We present in the following section this more general model with some theoretical results, which are based on parametric assumptions.\\
In Section 3  we derive some exact and asymptotic univariate and multivariate distributions of the parametric model.\\
Generalized extreme shock models can even be studied using nonparametric techniques. In Cirillo and H\"usler (2009) a nonparametric urn-based approach to extreme shock models is proposed. In Section 4 we briefly show how the same approach could be used to model the increasing threshold.\\
Section 5 concludes the paper.

\section{Extreme shock models with a possibly increasing  threshold} \label{model}
In some applications a system has a running-in period during which the critical load can increase and the structure is  strengthened because at the beginning the loads or shocks  are large but non-fatal.  This may happen in particular until the first damage or crack. After such an event the system can only be weakened. Such a pattern can be modeled as follows.\\
We let the arrival times $T_i$ of the shocks $X_i$ be, as mentioned, a
partial sum of \iid\ inter-arrival times $Y_j$ with distribution
$G$. The loads $X_i, i\ge 1$, are  an \iid\ sequence of r.v.'s with
distribution $F$. A shock or stroke $X_i$ is strengthening the
material if $X_i\in [\gamma,\beta)$. At the beginning the material
supports a maximal load  $\alpha$, the critical threshold. After a strengthening stroke,
the
 maximal load becomes larger, say $\alpha_1=\alpha+ b_1$ with $b_1 >0$. This boundary
increases with each strengthening stroke, inducing boundaries
$\alpha_j=\alpha + b_j$, $j\ge 1$ with $b_j\uparrow$. After the first harmful
stroke larger than $\beta$, but smaller than the critical level at this time point, the load
boundary decreases, because of possible cracks or some weakening
of the material. If it has reached the level $\alpha_k$ (because of  $k$
strengthening strokes before the first harmful stroke), the critical level becomes now $\alpha_k- c_1$, and
decreases further by the next harmful, nonfatal strokes to
$\alpha_k-c_2,\alpha_k-c_3,...$, with $c_j\uparrow (\ge 0)$. This is shown in Fig. \ref{pattern}. There might be an
upper load limit $\alpha^*$ for the $\alpha_k$, as well a lower
load limit $\alpha_*$ for the $\alpha_k-c_l$. We set that
$\alpha_*\ge \beta $. It is convenient to set $b_0=c_0=0$.
For notational reasons we define the number $N_-(n)$ of weakening shocks $X_i, i< n,$ before the $n$-th shock
$$N_-(n)=\sum_{i< n} 1(X_i\in [\beta, \alpha + b_{N_+(i)}-c_{N_-(i)}),  $$ with
$N_-(0)=0$ and  the  number $N_+(n)$ of
strengthening strokes $X_i,$ $i< n,$ before the $n$-th shock and  before the first weakening
or fatal shock  $$N_+(n)=\sum_{i< n} 1(X_i\in [\gamma, \beta), N_-(i)=0) ,$$ with $N_+(0)=0$.
Note that the critical boundary
for $X_i$ is $\alpha_i=\alpha + b_{N_+(i)}- c_{N_-(i)}$.\\
In addition, let $W$ be the index of the first harmful shock larger than $\beta$. It could indicate even a fatal shock being larger than the critical boundary at this time point. Hence, if such a shock occurs:
$$ W=\inf\{i: X_i\ge \beta\}\le \infty.$$ If the set is empty, we set $W=\infty$.
Hence, the shock $X_i$ has no impact if $X_i\le \gamma$; it induces a strengthening of the material if $i<W$ and $X_i\in [\gamma,\beta)$; and it is fatal, if  $X_i \ge  \alpha_i=\alpha + b_{N_+(i)}-c_{N_-(i)}$.
Note also, that $N_+(k)=N_+(W)$ for all $k\ge W(<\infty)$.
\begin{figure}[h]
\unitlength.91mm{\footnotesize
\begin{picture}(100,60)
\put(10,10){\line(1,0){82}}
\put(10,10){\line(0,1){45}}
\put(7,54){\makebox(0,0){$X_i$}}
\put(7,30){\makebox(0,0){$\gamma$}}
\put(7,40){\makebox(0,0){$\beta$}}
\put(7,45){\makebox(0,0){$\alpha$}}
\put(30,50){\makebox(0,0){$\alpha_i$}}
\put(25,6){\makebox(0,0){$i$}}
\put(90,6){\makebox(0,0){$\nu$}}
\put(90,9.5){\line(0,1){1}}
\put(10,30){\line(1,0){82}}
\put(10,40){\line(1,0){82}}
\put(10,45){\line(1,0){15}}
\put(25,48){\line(1,0){10}}
\put(35,50){\line(1,0){15}}
\put(50,46){\line(1,0){20}}
\put(70,43){\line(1,0){20}}
\put(10,0){
\multiput(12.5,10)(12.5,0){6}{\line(0,1){1}}
\put(25,7){\makebox(0,0){\footnotesize 10}}
\put(50,7){\makebox(0,0){\footnotesize 20}}
\put(75,7){\makebox(0,0){\footnotesize 30}}
\put(40,10){\line(0,1){1}}
\put(40,6){\makebox(0,0){$W$}}
\put(2.5,14){\circle{1}}
\put(5.0,11){\circle{1}}
\put(7.5,22){\circle{1}}
\put(10.0,25){\circle{1}}
\put(12.5,18){\circle{1}}
\put(15.0,34){\circle{1}}
\put(17.5,23){\circle{1}}
\put(20.0,19){\circle{1}}
\put(22.5,25){\circle{1}}
\put(25.0,37){\circle{1}}
\put(27.5,19){\circle{1}}
\put(30.0,28){\circle{1}}
\put(32.5,24){\circle{1}}
\put(35.0,16){\circle{1}}
\put(37.5,27){\circle{1}}
\put(40.0,42){\circle{1}}
\put(42.5,16){\circle{1}}
\put(45.0,32){\circle{1}}
\put(47.5,28){\circle{1}}
\put(50.0,24){\circle{1}}
\put(52.5,22){\circle{1}}
\put(55.0,29){\circle{1}}
\put(57.5,33){\circle{1}}
\put(60.0,44){\circle{1}}
\put(62.5,25){\circle{1}}
\put(65.0,32){\circle{1}}
\put(67.5,18){\circle{1}}
\put(70.0,20){\circle{1}}
\put(72.5,29){\circle{1}}
\put(75.0,27){\circle{1}}
\put(77.5,33){\circle{1}}
\put(80.0,45){\circle{1}}
}
\end{picture}
}
\caption{Realization of a sequence of shocks with strengthening and weakening load limits, depending on the values $X_i$. Here we have $\nu=32, W=16, N_+(\nu)= 2, N_-(\nu)= 2$.}\label{pattern}
\end{figure}
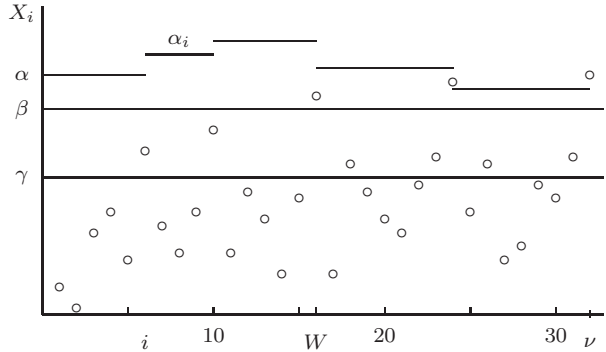
\\First we analyze in Section 3 the distribution of the number $\nu$ of shocks
until the first fatal shock:
$$ \nu=\min\{i: X_i\ge \alpha+ b_{N_+(i)}- c_{N_-(i)}\}.$$
where we use $b_0=0$ and $c_0=0$.
 The time until the fatal shock
is thus $ T_{\nu}$. Its distribution  depends on the distribution of $\nu$ and $G$.\\
We consider the asymptotic behaviour of these random
variables by letting the parameters $\alpha=\alpha(t),
\beta=\beta(t)$ and $ \gamma=\gamma(t) $ tend to $x_F\le\infty$ as $t\to
\infty$, where $x_F$ denotes the upper endpoint of the distribution $F$. We assume that $x_F$ is a continuity point of $F$. For the limit distributions, certain additional restrictions will be imposed also on
the $b_k$ and $c_k$, being also dependent on $t$. Hence
$\nu=\nu(t)$ will tend to $\infty$ in general, depending on the
underlying distribution $F$.


\section{The distribution of $\nu$}

The distribution of $\nu$ can be derived in this more general model as in the basic generalized extreme shock model of Gut and H\"usler (2005). First we derive the exact distribution and then analyze the asymptotic distributions which depend on the behaviour of the sequences $b_k$ and $c_k$.
For the derivation we use the notation $$\alpha_{k,l}=\alpha +b_k-c_l$$ for any $k,l\ge 0$.
\subsection{The exact distribution}
 To derive $P\{\nu> m\}$ we have to condition on the other random variables. If  $N_-(m+1)=0$, then $W> m$ and simply $P\{\nu>m, N_-(m+1)=0\}= F^m(\beta)$. If $l>0$ with $k<j$ and $m\ge j+l-1$,
we consider the joint distribution
\begin{eqnarray}&& \hspace*{-2cm}
P\{\nu>m, N_+(m)=k,
N_-(m+1)=l, W=j\}\nonumber
\\ &=&  { j-1 \choose k }
F^{j-1-k}(\gamma) \,[\bar{F}(\gamma)-\bar{F}(\beta)]^{k}\nonumber  \\ &&
\times {m-j\choose l-1} F^{m-j-l+1}(\beta)\,
\prod_{h=1}^l[\bar{F}(\beta)-\bar{F}(\alpha_{k,h-1}
)]\label{tau+N+N-}
\end{eqnarray}
or  for $m\ge j+l$ and $k<j$
\begin{eqnarray}&& \hspace*{-2cm}
P\{\nu=m, N_+(m)=k,
N_-(m)=l, W=j\}\nonumber
\\ &=&  { j-1 \choose k }
F^{j-1-k}(\gamma) \,[\bar{F}(\gamma)-\bar{F}(\beta)]^{k} \nonumber\\ &&
\times {m-j-1\choose l-1} F^{m-j-l}(\beta)\,
\prod_{h=1}^l[\bar{F}(\beta)-\bar{F}(\alpha_{k,h-1}
)]\bar{F}(\alpha_{k,l})\,. \label{tauN+N-}
\end{eqnarray}
If $k\ge j$ or $m\le j+l-1$, the latter probabilities are 0.\\
By summing the appropriate terms, we get  the exact univariate and multivariate distributions for $\nu, N_+(m), N_-(l)$ and $W$, as well as for  $N_+(\nu), N_-(\nu)$.
For instance, the joint distribution of  $N_+(\nu), N_-(\nu)$
with $l\ge 1$ is
\begin{eqnarray*}
&&\hspace*{-1cm} P\{ N_+(\nu)=k, N_-(\nu)=l\}= \sum_m  P\{ \nu=m, N_+(m)=k, N_-(m)=l\}
\\ &= & \sum_{m,j} {{j-1} \choose k} {{m-j-1} \choose {l-1}} F^{j-k-1}(\gamma) [\bar{F}(\gamma)-\bar{F}(\beta)]^k
F^{m-j-l}(\beta)\\ && \hspace*{2cm}\times  \prod_{h=0}^{l-1} \{1-\bar{F}(\alpha_{k,h})/\bar{F}(\beta)\} \bar{F}^l(\beta)
 \bar{F}(\alpha_{k,l})
\end{eqnarray*}
\begin{eqnarray}
&=&\sum_{j\ge k+1}{
 {j-1} \choose k} F^{j-k-1}(\gamma)\sum_{m\ge j+l}{{m-j-1} \choose {l-1}}F^{m-j-l}(\beta)\bar{F}^l(\beta)\nonumber \\
 &&\hspace*{2cm}\times  [\bar{F}(\gamma)-\bar{F}(\beta)]^k
 \prod_{h=0}^{l-1} \{1-\bar{F}(\alpha_{k,h})/\bar{F}(\beta)\}
 \bar{F}(\alpha_{k,l})\label{N+N-}
\end{eqnarray}
Note that the sums are summing all the probabilities of a negative binomial distribution, hence
\begin{eqnarray}
&&\hspace*{-1cm} P\{ N_+(\nu)=k, N_-(\nu)=l\}=
\sum_{h\ge 0}{{h+k} \choose k} F^{h}(\gamma)\nonumber \\&& \hspace*{1cm} \times  [\bar{F}(\gamma)-\bar{F}(\beta)]^k
\prod_{h=0}^{l-1} \{1-\bar{F}(\alpha_{k,h})/\bar{F}(\beta)\}
 \bar{F}(\alpha_{k,l})\nonumber \\
&=& \bar{F}^{-k-1}(\gamma)  [\bar{F}(\gamma)-\bar{F}(\beta)]^k
\prod_{h=0}^{l-1} \{1-\bar{F}(\alpha_{k,h})/\bar{F}(\beta)\}
 \bar{F}(\alpha_{k,l})\nonumber \\
&=&    [1-\bar{F}(\beta)/\bar{F}(\gamma)]^k
\prod_{h=0}^{l-1} \{1-\bar{F}(\alpha_{k,h})/\bar{F}(\beta)\}
[\bar{F}(\alpha_{k,l})/\bar{F}(\gamma)]\label{N+N-l}
\end{eqnarray}
For the case $l=0$ we get in the same way
\begin{eqnarray}
&&\hspace*{-2cm} P\{ N_+(\nu)=k, N_-(\nu)=0\}\nonumber\\&=& \sum_{m>k}  P\{ \nu=m, N_+(m)=k, N_-(m)=0, W=m\}
\nonumber\\
&= & \sum_{m>k}  {{m-1} \choose {k}} F^{m-k-1}(\gamma) [\bar{F}(\gamma)-\bar{F}(\beta)]^k
 \bar{F}(\alpha_{k,0})\nonumber \\
&=&    [1-\bar{F}(\beta)/\bar{F}(\gamma)]^k
 [\bar{F}(\alpha_{k,0})/\bar{F}(\gamma)]\label{N+N-0}
\end{eqnarray}
Other exact distributions can be derived in the same way by appropriate summation.
Sometimes, for the derivation of the asymptotic distributions we have to approximate these sums to simplify the formulas.\\
We give for later use the following exact distribution.
\begin{eqnarray} && \hspace*{-2cm}
P\{\nu>m, N_-(\nu)>0\}\nonumber \\&=& \sum_{k\ge 0,l>0,j\ge 1}P\{\nu>m, N_+(m)=k,
N_-(m+1)=l, W=j\}
\nonumber \\
&=& \sum_{k\ge 0,l>0,j\ge 1} { j-1 \choose k }
F^{j-1-k}(\gamma) \,[\bar{F}(\gamma)-\bar{F}(\beta)]^{k} \nonumber \\ &&
\times {m-j\choose l-1} F^{m-j-l+1}(\beta)\,
\prod_{h=1}^l[\bar{F}(\beta)-\bar{F}(\alpha_{k,h-1}
)].\label{tN_}
\end{eqnarray}
\vspace{5mm}

\subsection{The asymptotic  distribution}
For the asymptotic behaviour let $\alpha(t)\to\infty,
\beta(t)\to \infty$ and also $\gamma(t)\to \infty$ as
$t\to\infty$.
The asymptotic behaviour depends also on the assumptions of the sequences $b_k$ and $c_k$, which may depend also on  the parameter $t$.  But mostly we do not indicate the dependence on the parameter $t$. From the above  finite distributions, it is reasonable to use the conditions $\bar{F}(\beta(t))\to 0$ as $t\to\infty$,
\begin{equation}\lim_t \frac{\bar{F}(\beta(t))}{\bar{F}(\gamma(t))}= g\in
[0,1]\label{g}\end{equation}
and
\begin{equation}\lim_t \frac{\bar{F}(\alpha(t)+ b_k(t)-
c_l(t))}{\bar{F}(\beta(t))}= \lim_t \frac{\bar{F}(\alpha_{k,l}(t))}{\bar{F}(\beta(t))}= a_{k,l}\in [0,1].\label{a}\end{equation}
 Obviously, the  $a_{k,l}$ are monotone by the assumed monotonicity of the sequences $b_k$ and $c_k$, i.e., monotone decreasing in $k$ with $l$ fixed, and  monotone increasing in $l$ with $k$ fixed.
We consider only the interesting cases with $g, a_{k,l}\in (0,1)$.
The simplest case occurs if $a_{k,l}=a\in (0,1)$ for all $k,l$.
This implies that
$$\prod_{h=1}^l[\bar{F}(\beta)-\bar{F}(\alpha+
b_k-c_{h-1})]\sim [\bar{F}(\beta)(1-a)]^l
$$ for each $l\ge 1$.
But in  general we approximate
$$\prod_{h=1}^l[\bar{F}(\beta)-\bar{F}(\alpha+
b_k-c_{h-1})]\sim \bar{F}^l(\beta) \prod_{h=1}^l (1-a_{k,h})
$$ for each $l\ge 1$. We use for notational reason also $\prod_{h=1}^0 (1-a_{k,h})=1$.\\
To simplify the notation, we do not indicate the dependence on $t$ in the following, e.g., we write $\nu$ instead of $\nu(t)$.\\

\begin{thm}
{If (\ref{g}) and (\ref{a}) hold with $g, a_{k,l}\in[0,1]$, then for any $k\ge 0$ and $l\ge 0$
$$P\{N_+(\nu) =k, N_-(\nu)=l\}\to  g(1-g)^k \prod_{h=0}^{l-1}(1-a_{k,h})a_{k,l}$$ as $t\to\infty.$}
\end{thm}
\begin{proof}
Use (\ref{N+N-l}) for $l\ge 1$ and  (\ref{N+N-0}) for $l=0$, with  the assumptions (\ref{g}) and (\ref{a}) to derive immediately the claim, as $t\to\infty$.
\end{proof}
For the particular simple case that $a_{k,h}=a$ for all $k$ and $h$,
this limit distribution is the product of two geometric distributions
$$P\{ N_+(\nu)=k, N_-(\nu)=l\}\to g(1-g)^k (1-a)^l  a.$$
Because of the particular assumptions, the number of strengthening strokes does not have an influence on
the number of weakening  strokes asymptotically, which shows the asymptotic independence of $N_+(\nu)$ and $N_-(\nu)$ in this case.\\
\begin{thm} {If (\ref{g}) and (\ref{a}) hold with $a_{k,l}\in(0,1]$ for each $k,l$.
\\
i) Then for $k\ge 0=l$
$$
\lim_{t\to\infty}  P\{\nu \ge z/\bar{F}(\beta), N_+(\nu)=k, N_-(\nu)=0\} =\int_{z/g}^\infty v^k e^{-v}dv \; g\, (1-g)^k a_{k,0}/k!\,.$$
ii) For $l\ge 1$ and $k\ge 0$
\begin{eqnarray*}&& \hspace*{-1cm}
\lim_{t\to\infty} P\{\nu\ge z/\bar{F}(\beta), N_+(\nu)=k, N_-(\nu)=l\}
\\&=& \int_z^{\infty} \int_0^{1} y^k (1-y)^{l-1}\exp\{-y u(g^{-1}-1)\}
dy \; \exp\{-u\} \,u^{k+l} du \\
&&\hspace*{2cm} \times \, ((1-g)/g)^k\prod_{h=0}^{l-1}(1-a_{k,h})a_{k,l}/(k! (l-1)!)
 \end{eqnarray*}
 as $t\to\infty$.}
\end{thm}

\begin{proof}
Let  $k\ge0=l$ and $z_\beta=z/\bar{F}(\beta)$.
\begin{eqnarray*}&& \hspace*{-1cm}
P\{\nu\ge z_\beta, N_+(\nu)= k, N_-(\nu)=0\}=\\ &=&\sum_{m\ge z_\beta} P\{\nu = m, N_+(m)=k, N_-(m)=0\}\\
&=& \sum_{m=z_\beta}^{\infty} {m-1 \choose k} F^{m-k-1}(\gamma)[\bar{F}(\gamma)-\bar{F}(\beta)]^k
    \bar{F}(\alpha_{k,0})\\
&\sim & \sum_{m=z_\beta}^{\infty} \frac{m^k}{k!} F^{m}(\gamma)[\bar{F}(\gamma)-\bar{F}(\beta)]^k
    \bar{F}(\alpha_{k,0})\\
&\sim & \frac{1}{k!}\int_z^\infty v^k \exp\{-v\bar{F}(\gamma)/\bar{F}(\beta)\}\, dv\, (\bar{F}(\beta))^{-k-1}
  [\bar{F}(\gamma)-\bar{F}(\beta)]^k
  \bar{F}(\alpha_{k,0})
  \\
& \to & {\frac{1}{k!}}\int_{z/g}^\infty y^k e^{-y}dy \, (1-g)^k g\, a_{k,0}
\end{eqnarray*} as $t\to\infty$.  If $k=0$, then the integral is simply $e^{-z/g}$.

Now  let $l>0$.  By
(\ref{tauN+N-}) we have
\begin{eqnarray*}&& \hspace*{-1cm}
P\{\nu\ge z_\beta, N_+(\nu)=k, N_-(\nu)=l\}\\&=&\sum_{m\ge z_\beta} P\{\nu = m, N_+(m)=k, N_-(m)=l\}=\\
&=& \sum_{m=z_\beta}^{\infty}\sum_{j=k+1}^{m-l}{j-1\choose k}
  {m-j-1 \choose l-1} F^{j-k-1}(\gamma) F^{m-j-l}(\beta) \\
  && \hspace*{.8cm} \times \bar{F}^k(\gamma)(1-g+o(1))^k
  \bar{F}^{l+1}(\beta) \prod_{h=0}^{l-1}(1-a_{k,h}+o(1)) (a_{k,l}+o(1))
\end{eqnarray*}
Split the inner sum into three parts, with $j\le \epsilon m$, $\epsilon m< j< (1-\epsilon) m $ and
 $j\ge (1-\epsilon) m$ where we used $W=j$ in the summands.
The first  and the third sum are asymptotically negligible as $\epsilon \to 0$. This can be shown through derivations similar to the following ones for the second sum. In fact, for the second sum we have
 \begin{eqnarray*}&& \hspace*{-1cm}
P\{\nu\ge z_\beta, N_+(\nu)=k, N_-(\nu)=l, \epsilon m <W <(1-\epsilon)m\}=\\
&=& \sum_{m=z_\beta}^{\infty}\sum_{j>\epsilon m}^{(1-\epsilon)m} {j-1\choose k}{m-j-1
  \choose l-1}  F^{j-k-1}(\gamma) F^{m-j-l}(\beta)\\
 && \hspace*{.8cm} \times \bar{F}^k(\gamma)
  \bar{F}^{l+1}(\beta) (1-g+o(1))^k  \prod_{h=0}^{l-1}(1-a_{k,h}+o(1)) (a_{k,l}+o(1))
  \\
  \end{eqnarray*}
\begin{eqnarray*}
&\sim & \sum_{m=z_\beta}^{\infty}\sum_{j>\epsilon m}^{[(1-\epsilon)m]}
  \frac{(j-1)! (m-j-1)!}{k! (j-1-k)!(l-1)!(m-j-l)!}
  F^{j-k-1}(\gamma)  F^{m-j-l}(\beta)
\\
&&\hspace*{1.cm} \times \bar{F}^k(\gamma)
 \bar{F}^{l+1} (\beta)(1-g)^k \prod_{h=0}^{l-1}(1-a_{k,h}) a_{k,l}
\\
&\sim& \sum_{m=z_\beta}^{\infty}\sum_{j>\epsilon m}^{[(1-\epsilon)m]} \frac{j^k (m-j)^{l-1}}{k! (l-1)!}
   F^{j-k-1}(\gamma)  F^{m-j-l}(\beta)\\
  &&\hspace*{1.cm} \times\, \bar{F}^k(\gamma)
   \bar{F}^{l+1}(\beta) (1-g)^k\prod_{h=0}^{l-1}(1-a_{k,h}) a_{k,l} =:P_t
\end{eqnarray*}
Now we approximate the sums by integrals
\begin{eqnarray*}
 P_t&\sim& \sum_{m=z_\beta}^{\infty} \int_\epsilon ^{1-\epsilon} (y m)^k (m(1-y))^{l-1}F^{ym}(\gamma)F^{m(1-y)}(\beta)dy \,(m/k! (l-1)!)\\
&&\hspace*{2cm} \times\, \bar{F}^k(\gamma) \bar{F}^{l+1}(\beta) (1-g)^k\prod_{h=0}^{l-1}(1-a_{k,h}) a_{k,l}\\
&=&\sum_{m=z_\beta}^{\infty} \int_\epsilon ^{1-\epsilon} y^k (1-y)^{l-1}
  F^{ym}(\gamma)F^{m(1-y)}(\beta)dy \, m^{k+l}\\
 &&\hspace*{2cm} \times\, \bar{F}^k(\gamma)\bar{F}^{l+1}(\beta) (1-g)^k\prod_{h=0}^{l-1}(1-a_{k,h}) a_{k,l}/(k! (l-1)!)\\
&\sim& \int_z^{\infty} \int_\epsilon ^{1-\epsilon} y^k \,(1-y)^{l-1} F^{y u/\bar{F}(\beta)}
  (\gamma)F^{(1-y)u/\bar{F}(\beta)}(\beta)dy \,  u^{k+l} du \\
   &&\hspace*{2cm} \times  \, \bar{F}^k(\gamma)\bar{F}^{-k}(\beta)
   (1-g)^k\prod_{h=0}^{l-1}(1-a_{k,h}) a_{k,l} /(k! (l-1)!)\\
&\sim& \int_z^{\infty} \int_\epsilon ^{1-\epsilon} y^k (1-y)^{l-1}F^{y u/\bar{F}(\beta)}
  (\gamma)F^{(1-y)u/\bar{F}(\beta)}(\beta)dy \;u^{k+l} du \\
 &&\hspace*{2cm} \times\,
  g^{-k} (1-g)^k\prod_{h=0}^{l-1}(1-a_{k,h}) a_{k,l}/(k! (l-1)!)
\end{eqnarray*}
The factor $F^{y u/\bar{F}(\beta)}
(\gamma)F^{(1-y)u/\bar{F}(\beta)}(\beta)$
can be approximated for large $t$ by $$\exp(-(1+o(1))y u \bar{F}(\gamma)/\bar{F}(\beta) -(1+o(1))(1-y)u)\sim
\exp(-[yu(g^{-1} -1)-u])$$ uniformly
 for $u$ bounded. Hence by the dominated convergence we get
\begin{eqnarray*}
P_t
&\sim& \int_z^{\infty} \int_\epsilon ^{1-\epsilon} y^k (1-y)^{l-1}\exp\{-y u(g^{-1}-1)\}
dy \; \exp\{-u\} \,u^{k+l} du \\
&&\hspace*{1.5cm} \times \,((1-g)/g)^k\prod_{h=0}^{l-1}(1-a_{k,h}) a_{k,l}/(k! (l-1)!) \\
&\to& \int_z^{\infty} \int_0^{1} y^k (1-y)^{l-1}\exp\{-y u(g^{-1}-1)\}
dy \; \exp\{-u\} \,u^{k+l} du \\
&&\hspace*{1.5cm} \times \, ((1-g)/g)^k\prod_{h=0}^{l-1}(1-a_{k,h}) a_{k,l}/(k! (l-1)!)
\end{eqnarray*}
as $\epsilon \to 0$.
\end{proof}
If we set $z=0$, the integrals of both statements can be determined explicitly,  which implies the result of Theorem 1 in both cases $l=0$ and $l>0$.\\
Other limit distributions are determined by summing  the terms $P\{\nu=m, N_+(m)=k, N_-(\nu)=l, W=j\}$. Only under additional assumptions these distributions can be simplified. Let us deal with  such a particular case which generalizes Theorem 5.1  in Gut and H\"usler (2005). The formulas and the sums or integrals  can be simplified for instance, if the impact of the strengthening strokes is asymptotically negligible, i.e., when $a_{k,h}=a_h$, for all $k,h$. In this case we consider the limit distribution of
$\nu$ and $T_{\nu}$.\\

\begin{thm}{If (\ref{g}) and (\ref{a}) hold with $a_{k,h}=a_h\in[0,1]$  for each $k$ and $h\ge 1$. Then for $z_\beta=z/\bar{F}(\beta)$ with $z>0$
\begin{eqnarray*}P\{\nu>z_\beta \} &\to& \sum_{l=0}^\infty  \frac{z^l}{l!} e^{-z} \prod_{h=0}^{l-1} (1-a_h) \\&=&1-H(z)\end{eqnarray*} and
$$P\{ T_\nu > z_\beta\} \to 1-H(z/\mu)$$ as $t\to\infty$, where $\mu=E(Y_1)<\infty$.
}
\end{thm}

\begin{proof} The assumptions imply that $\prod_{h=0}^{l-1}(1-a_{k,h})=\prod_{h=0}^{l-1}(1-a_h)=:\tilde{a}_l$.
 Let  $\tilde{a}_0=1$.
\\ i) For the first statement we start with  (\ref{tN_}) and apply the assumptions
(\ref{g}) and (\ref{a}) to derive the limiting distribution.
\begin{eqnarray*}&&\hspace*{-1.5cm}
P\{\nu>m, N_-(m) >0\}\\&=&  \sum_{j=1}^m \sum_{l=1}^{m-j+1} \sum_{k=0}^{j-1} {j-1 \choose k }
F^{j-1-k}(\gamma) \,[\bar{F}(\gamma)-\bar{F}(\beta)]^{k} \\ &&
\hspace*{1cm}\times {m-j\choose l-1}
F^{m-j-l+1}(\beta)\,\prod_{h=1}^l[\bar{F}(\beta)-\bar{F}(\alpha_{k,h-1})]
\\
&\sim &   \sum_{j=1}^m \sum_{l=1}^{m-j+1} \sum_{k=0}^{j-1} {j-1 \choose k }
 F^{j-1-k}(\gamma) \,[\bar{F}(\gamma)-\bar{F}(\beta)]^{k} \\ &&
 \hspace*{1cm}\times {m-j\choose l-1}
 F^{m-j-l+1}(\beta)\,\bar{F}^l(\beta)\tilde{a}_l
\end{eqnarray*}
Now, using that the sum on $k$ can be simplified since it is a binomial sum, we get for this sum $(F(\gamma)+\bar{F}(\gamma)- \bar{F}(\beta))^{j-1}=(F(\beta))^{j-1}$.
\begin{eqnarray*}
&=&  \sum_{j=1}^m \sum_{l=1}^{m-j+1}{F}^{j-1}(\beta)
 {m-j \choose l-1} F^{m-j-l+1}(\beta) \bar{F}^l(\beta)\tilde{a}_l\\
&=& \sum_{l= 1}^m \sum_{j= 1}^{m-l+1}
 {m-j \choose l-1}  F^{m-l}(\beta)
 \bar{F}^l(\beta) \tilde{a}_l\\
 &=& \sum_{l= 1}^m {m \choose l}  F^{m-l}(\beta)
 \bar{F}^l(\beta) \tilde{a}_l \\
  &\to & \sum_{l= 1}^\infty \frac{z^l}{l!} \exp(-z ) \tilde{a}_l
 \end{eqnarray*}
using the normalization $m=z/\bar{F}(\beta)$, which tends to $\infty$.\\
ii) The second statement is immediate by applying the weak law of large numbers for $T_\nu/\nu \to \mu$, i.p., as $t\to\infty$.
\end{proof}
In other more general  cases we have to sum the terms of Theorem 2 to get the limit distribution of $\nu$ and $T_\nu$. Notice also that the dependence between $X_k$ and $Y_k$ has no influence on the limit distribution of $T_\nu$.

\section{A link to urn-based shock models}

A nonparametric approach to shock models has been recently proposed in Cirillo and H\"usler (2009, 2010). These nonparametric models are based on combinatorial processes, and in particular on combinations of Polya-like urn schemes.\\
Extreme shock models are modeled in Cirillo and H\"usler (2010) using a special version of the reinforced urn process of Muliere et al. (2000), that allows for a Bayesian nonparametric treatment of shock models.\\
Generalized extreme shock models are instead modeled in Cirillo and H\"usler (2009) by the means of a particular triangular Polya-like urn. This new model is called urn-based generalized extreme shock model (UbGESM). In the next lines we aim to show that a similar construction could also be applied, with some modifications, to the increasing threshold shock model we have introduced in Section \ref{model}. \\
For completeness, let us briefly recall the UbGESM.\\
The basic characteristic of the UbGESM construction is to get around the definition of the decreasing threshold mechanism of generalized extreme shock models, as developed in Gut and H\"usler (2005), by creating three different risk areas for the system (no risk or safe, risky and
default), by linking every area to a particular color and by working with the probability for the process to enter each area. If every time the process enters the risky area the probability of failing
increases, and this can be obtained with a triangular reinforcement matrix, such a modeling can be considered a sort of intuitive approach
to generalized extreme shock models. In some
sense, reinforcing the probability for the system to fail is like making the
risky threshold move down and vice-versa. \\
The authors consider an urn containing balls of three different colors: $x$ (safe), $y$ (risky),
and $w$ (default). The process evolves as follows.

\begin{enumerate}
\item At time $n$ a ball is sampled from the urn. The probability of sampling
a particular ball depends on the urn composition after time $n-1$;

\item According to the color of the sampled ball, the process enters (or remains in)
one of the three states of risk. For example, if the sampled ball is of type
$x$, the process is in a safe state, while it fails if the chosen ball is $w$;

\item The urn is then reinforced according to its reinforcement matrix RM (balanced and constant over time). It means that if the  sampled ball is of type $x$, then (it is replaced and) $\theta$ $x$-balls are added to the urn, if the sampled ball is of type $y$, then $\delta$ $y$-balls and $\lambda$ $z$-balls are added, and if the sampled ball is of type $z$, then $\theta$ $z$-balls are added.
\begin{equation}
RM=%
\begin{array}
[c]{c}%
x\\
y\\
w
\end{array}
\overset{%
\begin{array}
[c]{ccc}%
x & y & w
\end{array}
}{\left[
\begin{array}
[c]{ccc}%
1+\theta & 0 & 0\\
0 & 1+\delta & \lambda\\
0 & 0 & 1+\theta
\end{array}
\right]  },\text{ where }\lambda=\theta-\delta\label{urna}%
\end{equation}
\end{enumerate}

The distribution and the main properties of the urn process can be described analytically through the analysis of its generating function\footnote{The generating function of urn histories is a generating function of the form $H(z;x,y,w)=\sum_{n=0}^{\infty}f_n(x,y,w)\frac{z^{n}}{n!}$, where $f_n(x,y,w)$ is a counting function that counts the number of $x$, $y$, and $w$-balls in sampling sequences of length $n$. The gfuh is then that generating function, which enumerates all the possible compositions of the urn at time $n$, given its reinforcement matrix and initial composition.}, see the details in Cirillo and H\"usler (2009). In particular, we can quickly state the following theorem.

\begin{thm}[Cirillo and H\"usler (2009)]\label{moms}
Let $X_{n}$, $Y_{n}$ and $W_{n}$ represent the number of $x$, $y$ and $w$
balls in the urn at time $n$. Their moments show to be hypergeometric
functions, that is finite linear combinations of products and quotients of
Euler Gamma functions. In particular, the moments of order $l$ are given by

\begin{align*}
E\left[  (X_{n})_{l}\right]   &  =\theta^{l}\frac{\left(  \frac{a_{0}}{\theta
}\right)  ^{(l)}}{\left(  \frac{t_{0}}{\theta}\right)  ^{(l)}}n^{l}%
+O(n^{l-1}),\\
E\left[  (Y_{n})_{l}\right]   &  =\delta^{l}\frac{\left(  \frac{b_{0}}{\delta
}\right)  ^{(l)}}{\left(  \frac{t_{0}}{\theta
}\right)  ^{(l\frac{\delta}{\theta})}}n^{l\frac{\delta}{\theta}}%
+O(n^{(l-1)\frac{\delta}{\theta}}),\\
E\left[  (W_{n})_{l}\right]   &  =\lambda^{l}\frac{\left(  \frac{t_{0}-a_{0}}{\theta}{\
}\right)  ^{(l)}}{\left(  \frac{t_{0}}{\theta
}\right)  ^{(l\frac{\lambda}{\theta})}}n^{l\frac{\delta}%
{\theta}}+O(n^{(l-1)\frac{\delta}{\theta}}),
\end{align*}
where $t_{0}=a_{0}+b_{0}+c_{0}$, $\lambda=\theta-\delta$ and $\left(
\cdot\right)  ^{(n)}=\frac{\Gamma(x+n)}{\Gamma(x)}$ represents the standard Pochhammer formula. 
\end{thm}

Other results about the limit law of $X_n$, $Y_n$ and $W_n$ can be found in Cirillo and H\"usler (2009), together with results about the asymptotic exchangeability of the triangular urn process and its use from a Bayesian point of view.\\
The UbGESM shows to be very flexible and it is able to indirectly reproduce all the main results of Gut and H\"usler (2005). For example, computing the probability that $Y_{10}=b_{0}+1$ is like asking which is the probability for the system to overcome the risky threshold for the first time in $n=10$. In the same way, $P\left[  W_{n}=c_{0}+1\right]  $ represents the probability for the model to fail at time $n$.\\

\subsection{An urn model for the increasing threshold}
To model the possibly increasing threshold using the urn-based approach, we need to introduce a 4-color ($x,u,y,w$) urn, with initial composition ($a_0,d_0,b_0,c_0$). A possible reinforcement matrix\footnote{We could also think of a sacrificial urn (Flajolet et al., 2006), in which $w$-balls are removed every time a $u$-ball is sampled.} can be the following:
\begin{equation}
RM_2=%
\begin{array}
[c]{c}%
x\\
u\\
y\\
w
\end{array}
\overset{%
\begin{array}
[c]{cccc}%
x & u & y & w
\end{array}
}{\left[
\begin{array}
[c]{cccc}%
1+\theta & 0 & 0 &0\\
\theta & 1 &0&0\\
0 &0& 1+\delta & \lambda\\
0 & 0&0 &1+ \theta
\end{array}
\right]  }
\end{equation}
This matrix tells us that the balls of color $x,y,w$ behave as in the standard UbGESM, while balls $u$ represent the strengthening shock. In fact, every time a $u$-ball is sampled, the ball is not replaced in the urn and $\theta$ $x$-balls are added instead. For what concerns the process, it remains in the state it is actually visiting. In this way, some strengthening shocks may increase the probability of entering or remaining in the safe state, indirectly reproducing the increasing threshold mechanism for running-in. If we want all the $u$-balls to be removed sooner or later, in order to avoid further strengthening shocks, we can for example modify the second row of the $RM_2$ matrix, by changing $\theta$ with $\theta+1$ and $1$ with $0$.
\begin{thm}
Consider an urn process characterized by the reinforcement matrix $RM_2$ and with an initial composition
$(a_0,d_0,b_0,c_0)$ of balls. The 5-variables generating function of urn
histories is:
\begin{eqnarray*}
H(z;x,u,y,w)&=&x^{a_{0}}u^{d_0}y^{b_{0}}w^{c_{0}}(1-\theta x^{\theta}uz)^{-\frac{a_0+d_0
}{\theta}}(1-\theta w^{\theta}z)^{-\frac{c_{0}}{\theta}}\\&& \cdot\left(  1-y^{\delta
}w^{-\delta}\left(  1-(1-\theta w^{\theta}z)^{\frac{\delta}{\theta}}\right)
\right)  ^{-\frac{b_{0}}{\delta}}.
\end{eqnarray*}
\end{thm}
\begin{proof}
First notice that states $x,u$ are not directly dependent from states $y,w$ through reinforcement and vice versa. Hence matrix $RM_2$ can be seen as a combination of a records urn (Flajolet et al., 2006) for $x,u$ and a triangular urn for $y,w$. The result then comes from a direct application of the isomorphism theorem of Flajolet et al. (2005).
\end{proof}
Given the generating function of urn histories, we can then study the evolution of the balls in the urn as in Cirillo and H\"usler (2009).
\begin{thm}\label{moms}
Let $X_{n}$, $U_n$, $Y_{n}$ and $W_{n}$ represent the number of $x$, $u$, $y$ and $w$
balls in the urn at time $n$. Their moments of order $l$ are given by

\begin{align*}
E\left[  (X_{n})_{l}\right]   &  =\theta^{l}\frac{\left(  \frac{a_{0}+d_0}{\theta
}\right)  ^{(l)}}{\left(  \frac{t_{0}}{\theta}\right)  ^{(l)}}n^{l}%
+O(n^{l-1}),\\
E\left[  (U_{n})_{l}\right]   &  = d_0\\
E\left[  (Y_{n})_{l}\right]   &  =\delta^{l}\frac{\left(  \frac{b_{0}}{\delta
}\right)  ^{(l)}}{\left(  \frac{t_{0}}{\theta
}\right)  ^{(l\frac{\delta}{\theta})}}n^{l\frac{\delta}{\theta}}%
+O(n^{(l-1)\frac{\delta}{\theta}}),\\
E\left[  (W_{n})_{l}\right]   &  =\lambda^{l}\frac{\left(  \frac{t_{0}-a_{0}}{\theta}{\
}\right)  ^{(l)}}{\left(  \frac{t_{0}}{\theta
}\right)  ^{(l\frac{\lambda}{\theta})}}n^{l\frac{\delta}%
{\theta}}+O(n^{(l-1)\frac{\delta}{\theta}}),
\end{align*}
where $t_{0}=a_{0}+b_{0}+c_{0}+d_0$, $\lambda=\theta-\delta$.
\end{thm}
\begin{proof}
As said matrix $RM_2$ can be seen as a combination of a records urn for $x,u$ and a triangular urn for $y,w$. In particular notice that the number of $u$-balls does not vary over time.\\
For what concerns $X_n$, $Y_n$ and $W_n$, we only show the proof for $y$-balls, since the methodology is always the same.\\
Set $C_{n}=\frac{\Gamma(n+1)}{\left(
\frac{t_{0}}{\theta}\right)  ^{(n)}}$. Taking derivatives of the multivariate
generating function, one has%
\[
E[(Y_n)_l]=E\left[  Y_{n}(Y_{n}-1)\cdots(Y_{n}-l+1)\right]  =C_{n}\left[  z^{n}\right]
\frac{\partial^{l}H}{\partial y^{l}}|_{x=1,w=1},
\]
where $\left[  z^{n}\right]  $ represents the standard notation for the
operation of coefficient extraction (let $f(X)=\sum_{i=0}^{\infty
}u_{i}z^{i}$ be a generating function, then $\left[  z^{m}\right]
f(z)=\left[  z^{m}\right]  \sum_{i=0}^{\infty}a_{i}z^{i}=a_{m}$).\\
With some simple manipulations we have that%
\[
E\left[  (Y_{n})_{l}\right]  =C_n\left[\left(  \frac{b_{0}}{\delta}\right)
^{(l)}(1-\theta z)^{-\frac{(t_{0}+l\theta)\delta}{\theta}}+\left(  \frac
{b_{1}}{\delta}\right)  ^{(l)}(1-\theta z)^{-\frac{(t_{0}+(l-1)\theta)\delta
}{\theta}}+...\right]\newline%
\]\\
At this point, noting that for $\gamma_{1}<\gamma_{2}$, $\left[  z^{n}\right]
(1-z)^{-\gamma_{1}}=o(\left[  z^{n}\right]  (1-z)^{-\gamma_{2}})$, we discover that only the first term influences the
asymptotic behaviour. So, thanks to a coefficient extraction with respect to
$z$, we get the desired result.\\
\end{proof}
At this point, as shown in Flajolet et al. (2006) or in Cirillo and H\"usler (2009), one can study all the other properties of the urn process. Anyway such a study goes beyond the scope of the present section, whose aim is simply to build a bridge between generalized shock models with increasing threshold and urn-based shock models. It goes without saying that several different urn processes can be used to develop the alternative modeling.

\section{Conclusion}
We have proposed an extension of generalized extreme shock model by introducing a possibly increasing failure threshold.\\
While standard models assume that the crucial threshold for the system may only decrease over time, because of weakening shocks and obsolescence, we have assumed that, in particular at the beginning of the system's life, some strengthening shocks may increase the system tolerance to large shock, as it happens in running-in phases. This is for example the case of turbines' breaking-in in the field of engineering. However other fields of applications are easily identifiable.\\
On the basis of parametric assumptions, we have provided theoretical results and derived some exact and asymptotic univariate and multivariate distributions for the model.\\
In the last part of the paper we have also shown how to link the new model to some recent urn-based nonparametric approaches proposed in the literature (Cirillo and H\"usler, 2009).
\\
\\
\textbf{Acknowledgements:} The present work has been supported by the Swiss National Science Foundation.

\end{document}